\documentclass[11pt,reqno,a4paper]{amsart}

\usepackage{graphicx}
\usepackage{color}
\usepackage[USenglish]{babel}
\usepackage{latexsym}
\usepackage{amsmath}
\usepackage{amsfonts}
\usepackage{amssymb}
\usepackage{comment}
\usepackage{esint}
\usepackage[all]{xy}

\renewcommand{\to}{\rightarrow}
\newcommand{\pa}{\partial}

\newcommand{\ino}{\int_{\Omega}}

\renewcommand{\dfrac}{\displaystyle\frac}

\newcommand{\intbar}{\mathop{\int\makebox(-15.5,0){\rule[6pt]{.7em}{0.3pt}}\kern-6pt}\nolimits}

\newcommand{\ii}{\infty}

\newcommand{\dt}{\delta}

\newcommand{\al}{\alpha}

\newcommand{\sg}{\sigma}

\newcommand{\om}{\Omega}

\newcommand{\graf}[1]{\left\{\begin{array}{ll}#1\end{array}\right.}

\renewcommand{\a }{\alpha }

\newcommand{\D }{\Delta }

\newcommand{\lm }{\lambda }

\renewcommand{\O }{\Omega }

\def\p{\partial}

\newcommand{\be}{\begin{equation}}
\newcommand{\ee}{\end{equation}}
\newcommand{\beq}{\begin{equation}}
\newcommand{\eeq}{\end{equation}}

\newcommand{\R}{\mathbb{R}}

\newcommand{\N}{\mathbb{N}}

\newcommand{\dis}{\displaystyle}

\newcommand{\bns}{\alpha_{n}}

\newtheorem{theorem}{Theorem}[section]
\newtheorem{proposition}[theorem]{Proposition}
\newtheorem{definition}{Definition}[section]
\newtheorem{corollary}[theorem]{Corollary}
\newtheorem{remark}[theorem]{Remark}
\newtheorem{example}[theorem]{Example}

\newtheorem{lemma}[theorem]{Lemma}

\newcommand{\bpr}{\begin{proposition}}
\newcommand{\epr}{\end{proposition}}
\newcommand{\bex}{\begin{example}\rm}
\newcommand{\eex}{\end{example}}
\newcommand{\brm}{\begin{remark}\rm}
\newcommand{\erm}{\end{remark}}
\newcommand{\bdf}{\begin{definition}\rm}
\newcommand{\edf}{\end{definition}}
\newcommand{\bte}{\begin{theorem}}
\newcommand{\ete}{\end{theorem}}
\newcommand{\ble}{\begin{lemma}}
\newcommand{\ele}{\end{lemma}}
\newcommand{\bco}{\begin{corollary}}
\newcommand{\eco}{\end{corollary}}
\newcommand{\mycomment}[1]{}

\numberwithin{equation}{section}

\begin{document}
\title[A Harnack-type inequality]
{A Harnack-type inequality for a perturbed singular Liouville equation. }

\author{D. Bartolucci, P. Cosentino, L. Wu}

\address{Daniele Bartolucci, Department of Mathematics, University of Rome {\it ``Tor Vergata"} \\  Via della ricerca scientifica n.1, 00133 Roma, Italy. }
\email{bartoluc@mat.uniroma2.it}

\address{Paolo Cosentino, Department of Mathematics, University of Rome {\it ``Tor Vergata"} \\  Via della ricerca scientifica n.1, 00133 Roma, Italy. }
\email{cosentino@mat.uniroma2.it}

\address{Lina Wu, School of Mathematics and Statistics, Beijing Jiaotong University, Beijing, 100044, China}
\email{lnwu@bjtu.edu.cn}

\begin{abstract}
Motivated by the Onsager statistical mechanics description of turbulent Euler flows with point singularities, we obtain a Harnack-type inequality for sequences of solutions of the following perturbed Liouville equation,
\begin{equation}\nonumber
 -\D v_n=\left({\epsilon_n^2+|x|^2}\right)^{\alpha_n}V_n(x)e^{\displaystyle v_n} \qquad\text{in} \,\,\, \Omega,
\end{equation}
where $\epsilon_n\to0^+$, $\a_n\to\a_\infty\in(-1,1)$, $\O$ is a bounded domain in $\mathbb{R}^2$ containing the origin and $V_n$ satisfies,
\begin{equation}\nonumber
 0<a\leq V_n\leq b<+\ii, \,\, V_n\in C^{0}(\O), \,\,V_n\to V \;\; \text{locally uniformly in}\,\,{\O}.
\end{equation}
\end{abstract}



\maketitle
\section{Introduction}

We are concerned with the analysis of solutions of
\begin{equation}\label{intro:eq}
 -\D v_n=\left({\epsilon_n^2+|x|^2}\right)^{\alpha_n}V_n(x)e^{\displaystyle v_n} \qquad\text{in} \,\,\, \Omega,
\end{equation}
where $\epsilon_n\to0^+$, as $n\to+\infty$,
\begin{equation}
 \label{intro:alphan}
 \a_n\to\a_\infty\in(-1,+\ii),
\end{equation}
$\O$ is a bounded domain in $\mathbb{R}^2$ that contains the origin $x=0\in \om$ and $V_n$ satisfies,
\begin{equation}\label{intro: Vn}
 0<a\leq V_n\leq b<+\ii, \quad V_n\in C^{0}(\O), \quad V_n\to V \,\, \text{locally uniformly in}\,\,{\O}.
\end{equation}
If $\epsilon_n\equiv 0$ then \eqref{intro:eq} is just the classical Liouville equation (\cite{Lio}, \cite{Pic}), which has a long history in mathematics, see
for example \cite{bwz1}, \cite{clmp2}, \cite{kw}, \cite{lin-Lwang}, \cite{tar-sd}, \cite{yang} and references quoted therein.\\
On the other hand, equations of the form \eqref{intro:eq} have been discussed in \cite{det}, \cite{llty}, \cite{os}, under the condition
$\ino \left({\epsilon_n^2+|x|^2}\right)^{\alpha_n}V_n(x)e^{\displaystyle v_n} \leq C$, showing in general the
subtle phenomenon of ``blow up without concentration", where solutions blow up and, unlike the cases $\alpha_n \equiv 0$ (\cite{bm}) or either
$\epsilon_n\equiv 0$ (\cite{bt}), the right hand side of \eqref{intro:eq} does not converge to a sum of Dirac masses, the weak-$*$ limit being free instead
to develop a non vanishing absolutely continuous $L^1$ component.\\
Motivated by the analysis of the Onsager mean field statistical description of turbulent flows, in \cite{bcwyzOnsager} we recently considered \eqref{intro:eq},
where the weight $\left({\epsilon_n^2+|x|^2}\right)^{\alpha_n}$ encodes the presence of a fixed co-rotating ($\al_n<0$) or counter-rotating $(\al_n>0)$
vortex at the origin. In this context (see Lemma 3.2 in \cite{bcwyzOnsager}) a delicate point arise in the analysis of \eqref{intro:eq} which requires
an estimate of ``sup+C inf" type, in the same spirit of analogous results about the cases $\alpha_n \equiv 0$ (\cite{bls},\cite{CL},  \cite{S}, \cite{yy}) or
$\epsilon_n \equiv 0$ (\cite{b0}, \cite{b1}, \cite{b2}, \cite{b8}, \cite{BCLT}, \cite{cosentino2025}, \cite{Tar3}, \cite{T1}). Remark that
the examples of blow up without concentration in \cite{llty} are built upon the assumption $\al_n\to \al>1$. Also, blow up implies concentration as far as $\al_n\leq 1$ (see either Theorem 2.2 in \cite{bcwyzOnsager} or the general results in \cite{os}). 
Thus, by analogy with the ``regular" case $\al_n=0$ (\cite{S}), one may wonder about the validity of a ``sup+C inf" Harnack-type inequality as far as $\al_n\leq 1$. We point out that this class of ``sup+C inf" inequalities are the two-dimensional singular analogue of the classical ``sup x inf" inequalities first established in dimension $N\geq3$ in the context of the Yamabe problem, see \cite{lz}, \cite{Schoen} and references quoted therein. \\ \\ 
Let us recall some well known facts about this problem. 
It was first conjectured in \cite{bm} that if $v_n$ is a sequence of solution of \eqref{intro:eq} with $\alpha_n\equiv0$ and such that
\[
 0<a\leq V_n\leq b<+\infty,
\]
then, for any compact set $K\Subset\O$ there exists a constant $C_1\geq 1$, which depends only by $a$ and $b$ and a positive constant $C_2$,
which depends also by the distance dist$(K,\partial \O)$, such that,
\begin{equation}\label{sup+inf.intro}
 \underset{K}\sup \,v_n + C_1 \underset{\O}\inf\, v_n\leq C_2.
\end{equation}
Indeed \eqref{sup+inf.intro} was first proved in \cite{S}. Actually, by Remark 1.3 in \cite{b1}, 
we know that if $V_n$ also satisfies \eqref{intro: Vn}, then \eqref{sup+inf.intro} holds for any $C_1>1$, for a suitable $C_2$. Assuming that,
\begin{equation}\label{intro:KLip}
 ||\nabla V_n||_\infty\leq C_3,
\end{equation}
it has been proved in \cite{bls} that \eqref{sup+inf.intro} holds with $C_1=1$ and $C_2$ depending by $C_3$ as well.
Finally it has been proved in \cite{CL} that \eqref{sup+inf.intro} holds with $C_1=1$ under even weaker assumptions on $V_n$, such as
the existence of a logarithmic uniform modulus of continuity for $V_n$. Remark that in \cite{CL} one also finds that $C_1=\sqrt{\frac{b}{a}}$ is
the sharp constant for $V_n$ just satisfying $ 0<a\leq V_n\leq b<+\infty$.\\
Concerning the case $\al_\ii\neq 0$ and $\epsilon_n\equiv0$ (i.e. the so called conical singularities, see for example \cite{b8} for some details about the underlying 
geometrical problem),
these results were extended to cover the case $\alpha_\infty\in(-1,0)$ in \cite{b0}, \cite{b2} and more recently in \cite{cosentino2025} (see also \cite{b1}).
On the other hand, the case of positive exponents ($\alpha_\infty>0$ and $\epsilon_n\equiv0$) is more
delicate. It has been proved in \cite{T1} (see also \cite{tar-sd}) that  a weaker but still sharp inequality holds true: for any $\alpha_\infty>0$,
there exists a constant $C>0$ such that
\begin{equation}\label{u0+inf}
     v_n(0)+\underset{\O}\inf\, v_n\leq C
\end{equation}
for any sequence of solutions to 
\eqref{intro:eq}, \eqref{intro:alphan} and \eqref{intro:KLip}. See 
\cite{b0}, \cite{b1}, \cite{b8} for other
partial results about this problem. In particular in \cite{BCLT} the sharp inequality \eqref{sup+inf.intro} was obtained with $C_1=1$ but under stronger assumptions, 
see Theorem 1.3 in \cite{BCLT}.\\
Motivated by the Onsager vortex model pursued in \cite{bcwyzOnsager}, we are interested in a generalization of the result in \cite{b1}
to solutions of \eqref{intro:eq}, satisfying \eqref{intro:alphan} and \eqref{intro: Vn}, with $\epsilon_n\to0^+$ and $\alpha_\infty\in (-1,1)$.
Remark that any inequality of the form \eqref{sup+inf.intro} implies that blow up implies concentration in the sense of \cite{bm} and \cite{bt}. Therefore,
in view of the examples of blow up without concentration in \cite{llty}, it cannot hold in general as far as for $\al_\ii>1$.
Here we prove the following,
\begin{theorem}
 \label{intro:maintheo}
Let $\O$ be an open bounded domain in $\R^2$ which contains the origin, $\{0\}\subset \O$. Assume that
$v_n$ is a sequence of solutions of \eqref{intro:eq} satisfying \eqref{intro:alphan} and \eqref{intro: Vn}  with
$$
\alpha_\infty\in (-1,1).
$$
Then, for any
 $$ C_1>\max\left\{1,\tfrac{1+\alpha_\infty}{1-\alpha_\infty}\right\}$$
and for any compact set $K\subset\O$, there exists a constant $C_2>0$, which depends only by $a,b,dist(K,\p\O), \alpha_\infty$ and by the
uniform modulus of continuity of $V$ on $K$, such that,
\begin{equation}
 \label{intro:sup+Cinf}
 \underset{K}\sup\, v_n + C_1\underset{\O}\inf\, v_n\leq C_2.
\end{equation}
\end{theorem}
Interestingly enough, in the more subtle case $\al_\ii\in (0,1)$, the refined profiles obtained in \cite{bcwyzOnsager} suggest that \eqref{intro:sup+Cinf} is almost sharp, see section \ref{sec4} for further details. In particular those profiles provide examples 
of sequences of solutions of \eqref{intro:eq}, \eqref{intro:alphan}, \eqref{intro: Vn} that, for any fixed $C_1<\tfrac{1+\alpha_\infty}{1-\alpha_\infty}$, satisfy 
$$
\sup\limits_{K} v_n+C_1\inf\limits_{\om}v_n\to +\ii.
$$
Therefore, we see that, even in the case $\al_\ii\in (0,1)$, if the homogeneous classical expression of a conical singularity at the origin as $|x|^{2\al}$, is replaced  by some uniform approximating sequence as $(\epsilon_n^2+|x|^2)^{\al_n}$ in \eqref{intro:eq}, there is no chance to come up with an inequality neither of the form \eqref{u0+inf}. It is an interesting open problem to check whether or not, possibly under stronger assumptions as in \cite{BCLT}, the sharp constant in \eqref{intro:sup+Cinf} is $C_1=\tfrac{1+\alpha_\infty}{1-\alpha_\infty}$. We observe that a similar open problem was formulated in \cite{cosentino2025} in the case $\alpha_\ii\in(-1,0)$ and $\epsilon_n\equiv0$, under weaker assumptions on $V_n$.

\bigskip 

Concerning the proof, we argue as in \cite{b1} (see also \cite{S}) via known blow-up arguments (\cite{bt}, \cite{det}, \cite{ls}, \cite{llty}), where one has
to carefully handle the various possibilities arising due to the perturbation (i.e. $\epsilon_n\to 0^+$) of the homogeneous conical singularity.
The contradiction argument requires, in each one of the  blow up scenarios at hand, to find at least one bubble coming with some large enough amount of mass.
This is also why we need a preliminary result, which is a ``minimal mass" Lemma for solutions of \eqref{intro:eq}, see Section \ref{sec2}.

\bigskip

This paper is organized as follows: in Section 2 we give a proof of the minimal mass Lemma, while the proof of the ``$\sup+C\inf$''
inequality is done in Section 3. We discuss the sharpness of the ``$\sup+C\inf$'' inequality in section \ref{sec4}. Lastly, in the appendix, we state some well-known properties of the solutions of some perturbed singular
Liouville equations in $\R^2$.

\medskip

\section{A Minimal Mass Lemma}\label{sec2}
Let us consider a solution sequence $v_n$ of the problem:
\begin{equation}
 \label{eqbase}
 \graf{
 -\D v_n=\left({\epsilon_n^2+|x|^2}\right)^{\alpha_n}V_n e^{\dis v_n} \quad \text{in}\quad B_1,\\
 \hfill \\
 \int_{B_1}\left({\epsilon_n^2+|x|^2}\right)^{\alpha_n} e^{\dis v_n}\leq C }
\end{equation}
with $\epsilon_n\to0^+$,
\begin{equation}
 \label{alpha_n}
  \alpha_n\to\alpha_\infty\in(-1,1),\,\,\text{as}\,\,n\to+\infty
\end{equation}
and assume that $V_n$ satisfies,
\begin{equation}
 \label{V+}
 V_n\geq 0, \,\, V_n\in C^{0}(\overline B_1),
\end{equation}
\begin{equation}
 \label{Vconvergence1}
 V_n\to V \quad \text{uniformly in}\,\,\overline{B}_1.
\end{equation}

We assume that there exists a sequence of points $\{x_n\}\subset B_1$ such that
\begin{align}
\label{blowup:hyp}
 x_n\to 0 \,\,\,\,\text{and}&\,\,\,\, \underset{B_1}\sup\, v_n= v_n(x_n)\to +\infty, \,\,\text{as}\,\, n\to +\infty.
\end{align}

Here and in the rest of this paper we will often pass to subsequences which will not be relabelled.

\bigskip
\begin{lemma}[Minimal Mass Lemma]\label{minimalmasslemma}$\,$\\
Let $v_n$ be a sequence of solutions of \eqref{eqbase} satisfying \eqref{blowup:hyp}, where $V_n$ and $V$ satisfy \eqref{V+}
and \eqref{Vconvergence1}. Assume furthermore that \eqref{alpha_n} is satisfied and
\begin{equation*}
 \left({\epsilon_n^2+|x|^2}\right)^{\alpha_n}V_n e^{\dis v_n} \rightharpoonup \eta \qquad \text{weakly in the sense of measures in}\, B_1,
\end{equation*}
for some bounded Radon measure $\eta$, where we define,
$$
\beta:=\eta(\{0\}).
$$
Then $V(0)>0$ and,
\begin{align}
 \label{minimalmass}
\beta\geq\begin{cases}
            8\pi(1+\alpha_\infty), &\text{if}\,\,\al_\ii<0, \\
            8\pi, &\text{if}\,\,\al_\ii\geq 0.
            \end{cases}
\end{align}
\end{lemma}

\begin{proof}
 First of all, as far as \eqref{minimalmass} is concerned, it is enough to show that, for some $r\in(0,1)$,
\begin{align}
 \label{minimalmass2negative}
\underset{n\to\infty}\liminf \int_{B_r(0)}\left({\epsilon_n^2+|x|^2}\right)^{\alpha_n}V_n e^{\dis v_n} \geq\begin{cases}
            8\pi(1+\alpha_\infty), &\text{if}\,\,\al_\ii<0, \\
            8\pi, &\text{if}\,\,\al_\ii\geq 0.
            \end{cases}
\end{align}
Let us fix $r\in(0,1)$ and, for $n$ large enough, by \eqref{blowup:hyp} assume without loss of generality that,
$$
\underset{\overline B_r}\sup\, v_n= v_n(x_n)\to +\ii \;\quad \mbox{and}\;\quad |x_n|\to 0.
$$
Here and in the rest of this proof we set
\[
 \delta_n^{2(1+\bns)}=e^{-\dis {v_n(x_n)}}\to 0,\,\, \text{as} \,\, n\to\infty,
\]
and
\[
 t_n=\max\{|x_n|,\delta_n\}\to 0,\,\, \text{as} \,\, n\to\infty.
\]

Here, we are naturally led to analyze three different cases:
\begin{itemize}
 \item Case (I): there exists a subsequence such that $\tfrac{\epsilon_n}{t_n}\to+\infty$, as $n\to+\infty$;
 \item Case (II): there exists a constant $C_1>0$ such that
\begin{equation}
 \label{Hypothesis:CaseII}
 \frac{\epsilon_n}{t_n}\leq C_1,\,\,\text{for all}\,\,n\in\N;
\end{equation}
 and, possibly along a subsequence, $\tfrac{|x_n|}{\delta_n}\to+\infty$, as $n\to+\infty$,
 \item Case (III): there exists a constants $C_1>0$ and $C_2>0$ such that
 \begin{equation}
 \label{Hypothesis:CaseIII}
 \frac{\epsilon_n}{t_n}\leq C_1,\quad \tfrac{|x_n|}{\delta_n}\leq C_2,\,\,\text{for all}\,\,n\in\N.
\end{equation}
\end{itemize}

We start with Case (I) and notice that,
\begin{equation*}
 \frac{\epsilon_n}{\delta_n}\to+\infty\qquad\text{and}\qquad \frac{\epsilon_n}{|x_n|}\to+\infty.
\end{equation*}
Let us define
\begin{equation*}
B^{(n)}=B_{\frac{r}{\epsilon_n}}(0),
\end{equation*}
and
\begin{equation*}
 w_n(z)=v_n(\epsilon_n z)+2(1+\bns)\log\epsilon_n\quad\text{in}\quad B^{(n)},
\end{equation*}
which satisfies
\begin{align*}
\begin{cases}
-\D w_n=W_n(z)e^{\dis w_n} \quad\text{in}\quad B^{(n)},\\
W_n(z)=\left({1+|z|^2}\right)^{\alpha_n} V_n(\epsilon_n z)\to
V(0)\left({1+|z|^2}\right)^{\alpha_\infty}\;\mbox{in}\;C^t_{loc}(\R^2),\\
\int\limits_{B^{(n)}}\left({1+|z|^2}\right)^{\alpha_n}e^{\dis w_n}\leq C,\\
w_n(\tfrac{x_n}{\epsilon_n})=2(1+\alpha_n)\log(\tfrac{\epsilon_n}{\delta_n})\to +\infty.
\end{cases}
\end{align*}
In particular, we notice that $\tfrac{x_n}{\epsilon_n}\to 0$ and also that,
\[
 \int_{B_{R}(0)}e^{\dis w_n}\leq \tilde C,
\]
for any $R>0$ and some constant $\tilde C>0$. Thus, by the concentration-compactness alternative
for regular Liouville-type equations
on $B_{R}(0)$ (see \cite{bm} and \cite{ls}), we see that $V(0)>0$ and
\[
 \underset{n\to\infty}\liminf \int_{B_{R}(0)}W_n(z)e^{\dis w_n}=8m\pi,
\]
where $m\in\N^+$, so that, for any $r\in (0,1)$, we have that
\[
 \underset{n\to\infty}\liminf \int_{B_r(0)}\left({\epsilon_n^2+|x|^2}\right)^{\alpha_n}V_n e^{\dis v_n} =
 \underset{n\to\infty}\liminf \int_{B^{(n)}}\left({1+|z|^2}\right)^{\alpha_n}V_n(\epsilon_n z)e^{\dis w_n} \geq
\]
\[
 \geq \underset{n\to\infty}\liminf \int_{B_{R}(0)}\left({1+|z|^2}\right)^{\alpha_n}V_n(\epsilon_n z)e^{\dis w_n}=8\pi m,
\]

which implies \eqref{minimalmass2negative}, thereby concluding the proof in Case (I).
\medskip

Next, we consider Case (II), where we notice that,
\begin{equation}
 \label{epsilonxn}
 \frac{\epsilon_n}{|x_n|}\leq C
\end{equation}
for some $C>0$ constant. Indeed, in this case we have,
\begin{align*}
 \frac{\epsilon_n}{|x_n|}=\frac{\epsilon_n}{t_n}\frac{t_n}{\delta_n}\frac{\delta_n}{|x_n|}=\begin{cases}
         \frac{\epsilon_n}{t_n}\frac{\delta_n}{|x_n|},&\quad\text{if}\,\,\delta_n\geq|x_n|,\\
         \frac{\epsilon_n}{t_n}, &\quad\text{if}\,\,\delta_n\leq|x_n|.                                                                                    \end{cases}
\end{align*}
Then, besides
\[
 \frac{|x_n|}{\delta_n}\to+\infty,\,\,\text{as}\,\,n\to+\infty,
\]
possibly along a subsequence, we also have that,
\[
 \frac{\epsilon_n}{|x_n|}\to\bar \epsilon_0,\quad \frac{x_n}{|x_n|}\to \bar z_0,\;\;|\bar z_0|=1,\quad \text{as}\,\,n\to+\infty,
\]
for some $\bar \epsilon_0 \geq 0$. In this situation we define,
\begin{equation}\label{dn1}
D^{(n)}=B_{\frac{r}{|x_{n}|}}(0),
\end{equation}
\begin{equation}\label{barwn:IB}
 \bar w_{n}(z)=v_n(|x_{n}| z)+2(1+\alpha_n)\log(|x_{n}|)\quad\text{in}\quad D^{(n)},
\end{equation}
which satisfies
\begin{align*}
\begin{cases}
-\D \bar w_{n}=\bar W_{n}(z)e^{\dis \bar w_{n}} \quad\text{in}\quad D^{(n)},\\
\bar W_{n}(z)=\left({\frac{\epsilon_n^2}{|x_{n}|^2}+|z|^2}\right)^{\alpha_n}V_n(|x_{n}| z)
\to V_n(0)(\bar \epsilon_0^2+|z|^2)^{\alpha_\infty}\;\mbox{in}\;C^t_{loc}(\R^2),\\
\int\limits_{D^{(n)}}\left({\frac{\epsilon_n^2}{|x_{n}|^2}+|z|^2}\right)^{\alpha_n} e^{\dis \bar w_n}\leq C,\\
\bar w_{n}(\tfrac{x_{n}}{|x_{n}|})=2(1+\alpha_n)\log(\tfrac{|x_{n}|}{\delta_{n}})\to +\infty,
\end{cases}
\end{align*}
 Again, we notice that,
\[
 \int_{B_{R}(\bar z_0)}e^{\dis \bar w_n}\leq \bar C,
\]
for some $R>0$ small enough and some constant $\bar C>0$. Therefore, again by the concentration-compactness
alternative for regular Liouville-type equations on $B_{R}(\bar z_0)$ (see \cite{bm} and \cite{ls}), we see that
$(\bar \epsilon_0^2+1)^{\al_{\ii}}V(0)>0$ and
\[
 \underset{n\to\infty}\liminf \int_{B_{R}(\bar z_0)}\bar W_{n}(z)e^{\dis \bar w_n}=8\pi m,
\]
where $m\in\N^+$. As a consequence we have that,
\[
 \underset{n\to\infty}\liminf \underset{B_r(0)}\int \left({\epsilon_n^2+|x|^2}\right)^{\alpha_n}V_n e^{\dis v_n} = \underset{n\to\infty}\liminf \underset{D^{(n)}}
 \int \left({\frac{\epsilon_n^2}{|x_n|^2}+|z|^2}\right)^{\alpha_n}V_{n}(|x_n|z)e^{\dis \bar w_n}
\]
\[\geq\underset{n\to\infty}\liminf \int_{B_{R}(0)} \left({\frac{\epsilon_n^2}{|x_n|^2}+|z|^2}\right)^{\alpha_n}V_{n}(|x_n|z)
e^{\dis \bar w_n}=8\pi m.
\]
This is \eqref{minimalmass2negative}, which concludes the study of Case (II).

\medskip
At last, we examine Case (III). We notice that \eqref{Hypothesis:CaseIII} implies that,
\begin{equation*}
 \frac{\epsilon_n}{\delta_n}\leq C,
\end{equation*}
for some $C>0$ constant, since in fact we have,
\begin{align*}
 \frac{\epsilon_n}{\delta_n}=\frac{\epsilon_n}{t_n}\frac{t_n}{|x_n|}\frac{|x_n|}{\delta_n}=\begin{cases}
         \frac{\epsilon_n}{t_n},&\quad\text{if}\,\,\delta_n\geq|x_n|,\\
         \frac{\epsilon_n}{t_n}\frac{|x_n|}{\delta_n}, &\quad\text{if}\,\,\delta_n\leq|x_n|.                                                                                    \end{cases}
\end{align*}
Let us define
$$
D_n=B_{\frac{r}{\delta_n}}(0),
$$
and possibly along a subsequence, assume without loss of generality that,
\begin{equation*}
 \frac{x_n}{\delta_n}\to y_0, \,\,\text{as}\,\,n\to+\infty
\end{equation*}
 and
\begin{equation*}
 \frac{\epsilon_n}{\delta_n}\to \epsilon_0, \,\,\text{as}\,\,n\to+\infty,
\end{equation*}
for some $y_0\in\R^2$ and $\epsilon_0\geq 0$. At this point, we define,
\begin{equation*}
 \widetilde w_n(y)=v_n(\delta_n y)+2(1+\alpha_n)\log \delta_n=v_n(\delta_n y)-v_n(x_n)\quad\text{in}\,\, D_n,
\end{equation*}
which satisfies,
\begin{align}\label{equationgoodcase}
\begin{cases}
 -\D \widetilde w_n=\left({\frac{\epsilon_n^2}{\delta_n^2}+\left|y\right|^2}\right)^{\alpha_n}
 V_n(\delta_n y)e^{\dis \widetilde w_n}:=f_n \qquad\text{in}\,\,D_n, \\
\int_{D_n}\left({\frac{\epsilon_n^2}{\delta_n^2}+|y|^2}\right)^{\alpha_n}e^{\dis \widetilde w_n}\leq C,\\
\widetilde w_n(y)\leq\widetilde w_n(\tfrac{x_n}{\delta_n})=\max\limits_{D_n}\widetilde w_n=0.
\end{cases}
\end{align}
Remark that $f_n$ is uniformly bounded in $L^{p}(B_R)$, for $R$ large enough where $p=p(|\sigma|)>1$, if $\sigma<0$, while $p=\infty$, if $\sigma\geq0$.
Moreover we see that, by the Harnack inequality, for every $R\geq 1+ 2|y_0|$ there exists a constant $C_R>0$ such that
\begin{equation}
 \label{Estimateboundary}
  \sup_{\p B_{R}}|\widetilde w_n|\leq C_R.
\end{equation}
Indeed, since $f_n$ is uniformly bounded in $L^p(B_{2R})$, for some $p>1$, and
$$\sup_{\p B_{2R}}\widetilde w_n\leq \sup_{ B_{4R}}\widetilde w_n =\widetilde w_n(\tfrac{x_n}{\delta_n})=0,$$
then there exist $\tau\in(0,1)$ and $C_0>0$, which does not depend by $n$, such that
\[
 \sup_{ B_{R}} \widetilde w_n\leq \tau\inf_{B_{R}}\widetilde w_n +C_0
\]
and this implies that
\[
 \inf_{\p B_{R}}\widetilde w_n=\inf_{ B_{R}}\widetilde w_n\geq -\tau^{-1}C_0,
\]
where we used the fact $\sup\limits_{ B_{R}} \widetilde w_n=0$, for any $n$ large enough, and the superharmonicity of $\widetilde w_n$.\\
Therefore, by \eqref{equationgoodcase}, \eqref{Estimateboundary} and standard elliptic estimates, we conclude that $\widetilde w_n$ is uniformly
bounded in $C^{t}_{\rm loc}(\R^2)$, for some $t\in(0,1)$ and then, possibly along a subsequence, $\widetilde w_n\to \widetilde w$
uniformly on compact sets of $\R^2$, where $\widetilde w$ satisfies
\begin{align*}
\begin{cases}
 -\D \widetilde w=(\epsilon_0^2+|y|^2)^{\alpha_\infty}V(0)e^{\dis \widetilde w} \quad\text{in}\quad\R^2 \\
\int_{\R^2}(\epsilon_0^2+|y|^2)^{\alpha_\infty}e^{\dis \widetilde w}\leq C
\end{cases}
\end{align*}
Since $\widetilde w$ is bounded from above, then necessarily $V(0)\neq0$. At this point, by Lemma \ref{massstrange} below, we have that either,

\begin{align*}
\underset{\R^2}\int(\epsilon_0^2+|y|^2)^{\alpha_\infty}V(0)e^{\dis \widetilde w}\geq 8\pi(1+\alpha_\infty), \quad&\text{if}\quad \alpha_\infty<0,
\end{align*}
or,
\begin{align*}
\underset{\R^2}\int(\epsilon_0^2+|y|^2)^{\alpha_\infty}V(0)e^{\dis \widetilde w}> \max\{8\pi,4\pi(1+\alpha_\infty)\}\geq 8\pi, \quad &\text{if}\quad \alpha_\infty\geq 0.
\end{align*}
Therefore we have,
\[
 \underset{n\to\infty}\liminf \underset{B_r(0)}\int \left({\epsilon_n^2+|x|^2}\right)^{\alpha_n}V_n e^{\dis v_n} =
 \underset{n\to\infty}\liminf \underset{D_n}\int \left({\frac{\epsilon_n^2}{\delta_n^2}+\left|y\right|^2}\right)^{\alpha_n}
 V_n(\delta_n y)e^{\dis \widetilde w_n}\geq\underset{\R^2}\int (\epsilon_0^2+|y|^2)^{\alpha_\infty}V(0)e^{\dis \widetilde w}
 \]
which implies that in this case we have either,
\begin{align*}
\underset{n\to\infty}\liminf \underset{B_r(0)}\int\left({\epsilon_n^2+|x|^2}\right)^{\alpha_n}V_n e^{\dis v_n} \geq 8\pi(1+\alpha_\infty),
\quad&\text{if}\quad \alpha_\infty<0,
\end{align*}
or
\begin{align*}
\underset{n\to\infty}\liminf \underset{B_r(0)}\int \left({\epsilon_n^2+|x|^2}\right)^{\alpha_n}V_n e^{\dis v_n} >
\max\{8\pi,4\pi(1+\alpha_\infty)\}\geq8\pi, \quad &\text{if}\quad \alpha_\infty\geq 0.
\end{align*}
Therefore we readily verify that \eqref{minimalmass2negative} holds in Case (III) as well, which concludes the proof.
\end{proof}

\bigskip
\bigskip

\section{A ``Sup+ C Inf'' estimate}

In this section, we prove Theorem \ref{intro:maintheo}.
\begin{proof}
Without loss of generality, we can assume that $K\subset B_2(0)\Subset\O$. We argue by contradiction and assume that there exist
$v_n$ and $V_n$ which satisfy the hypotheses of Theorem \ref{intro:maintheo} such that,
\begin{equation}
 \label{contradiction:sup+Cinf}
 \underset{K}\sup\, v_n + C_1\underset{\O}\inf\, v_n\to+\infty,
\end{equation}
for some
$$
C_1>\max\left\{1,\tfrac{1+\alpha_\infty}{1-\alpha_\infty}\right\}.
$$
Let $y_n\in K$ such that $\sup_K\,v_n=v_n(y_n)$. By \eqref{contradiction:sup+Cinf}, we have that $v_n(y_n)\to+\infty$ and, up to a subsequence, we can assume that $y_n\to y_0$.
If $y_0\neq 0$, then $\hat V_n(x)=(\epsilon_n^2+|x|^2)^{\alpha_n}V_n(x)$ satisfies \eqref{intro: Vn} in a small enough neighborhood of $y_0$ and
the contradiction follows from Remark 1.3 in \cite{b1}.\\
Hence, we are left with the case $y_0=0$. Let us define, for $r\in(0,1]$ and $n\in\N$, the following function
\[
 \psi_n(r;\beta):=v_n(y_n)+\tfrac{C_1}{2\pi r}\int_{\p B_r(y_n)}v_n\,d\sigma+2(1+C_1)(1+\beta)\log r.
\]
By using the divergence theorem, we see that,
\begin{align*}
 \tfrac{\p \psi_n}{\p r}(r;\alpha_n)&=\tfrac{C_1}{2\pi r}\int_{ B_r(y_n)}\D v_n\,dx+2(1+C_1)(1+\alpha_n)\tfrac{1}{r}\\
 &=-\tfrac{C_1}{2\pi r}\int_{B_r(y_n)}(\epsilon_n^2+|x|^2)^{\alpha_n}V_ne^{v_n}\,dx+2(1+C_1)(1+\alpha_n)\tfrac{1}{r},
\end{align*}
for every $r\in(0,1]$ and for every $n\in\N$. Thus we have
\[
 \tfrac{\p \psi_n}{\p r}(r;\alpha_n)\geq 0\qquad\text{if and only if}\qquad \int_{ B_r(y_n)}(\epsilon_n^2+|x|^2)^{\alpha_n}V_ne^{v_n}\,dx\leq 4\pi(1+\alpha_n)\tfrac{C_1+1}{C_1}.
\]
Since $C_1>\max\{1,\tfrac{1+\alpha_\ii}{1-\alpha_\ii}\}$, for every $n$ large enough, $C_1>\max\{(1+2\tfrac{\alpha_\ii-\alpha_n}{1+\alpha_n})^{-1}, \tfrac{1+\alpha_n}{1-\alpha_n}\}$ and
\begin{equation}
\label{limitpi}
 4\pi(1+\alpha_n)\tfrac{C_1+1}{C_1}<\min\{8\pi,8\pi(1+\alpha_\ii)\},
\end{equation}
Moreover, since $C_1>1$, then
\begin{equation}
\label{limitpi2}
 4\pi\tfrac{C_1+1}{C_1}<8\pi.
\end{equation}
Now, we define
\[
 s_n:=\sup\left\{r\in(0,1]\,\Big|\,\int_{ B_r(y_n)}(\epsilon_n^2+|x|^2)^{\alpha_n}V_ne^{v_n}\,dx\leq 4\pi(1+\alpha_n)\tfrac{C_1+1}{C_1}\right\}
\]
and
\[
 \sigma_n:=\sup\left\{r\in(0,1]\,\Big|\,\int_{ B_r(y_n)}(\epsilon_n^2+|x|^2)^{\alpha_n}V_ne^{v_n}\,dx\leq 4\pi\tfrac{C_1+1}{C_1}\right\}.
\]
Hence, we deduce that $\psi_n(r;\alpha_n)$ and $\psi_n(r;0)$ have respectively a unique point of maximum with
\[
 \psi_n(1;\alpha_n)\leq \psi_n(s_n;\alpha_n)
\qquad\text{and}\qquad
\psi_n(1;0)\leq \psi_n(\sigma_n;0).
\]
Also, by \eqref{limitpi}, we have that, for $n$ large enough,
\begin{equation*}
 \int_{ B_{s_n}(y_n)}(\epsilon_n^2+|x|^2)^{\alpha_n}V_ne^{v_n}\,dx<\min\{8\pi,8\pi(1+\alpha_\infty)\}.
\end{equation*}
On the other hand, by \eqref{limitpi2}, we have that
\begin{equation*}
 \int_{ B_{\sigma_n}(y_n)}(\epsilon_n^2+|x|^2)^{\alpha_n}V_ne^{v_n}\,dx<8\pi.
\end{equation*}
Hence,
\begin{align*}
 \sup_K\,v_n+C_1\inf_{\O}\,v_n&\leq v_n(y_n)+C_1\inf_{B_1(y_n)}\,v_n\leq \psi_n(1;\alpha_n)\\
 &\leq \psi_n(s_n;\alpha_n)\leq (1+C_1)v_n(y_n)+2(1+C_1)(1+\alpha_n)\log s_n
\end{align*}
where we used the superhamonicity of $v_n$. Thus, by using \eqref{contradiction:sup+Cinf} we deduce that,
\begin{equation}
 \label{blowupsequenceatsn}
 v_n(y_n)+2(1+\alpha_n)\log s_n\to+\infty.
\end{equation}
The same argument shows that,
\begin{equation}
 \label{blowupsequenceatsigman}
 v_n(y_n)+2\log \sigma_n\to+\infty.
\end{equation}
Without loss of generality, we can assume that $s_n\to0^+$. Indeed, assume by contradiction that
$\limsup_{n\to\infty}s_n>0$, then there exists a subsequence of $s_{n}$,
which we do not relabel, such that $0<2\bar s \leq s_{n}\leq 1$, for every $n$ large enough.
In particular, by passing to a further subsequence if necessary,
we would also have that,
\begin{equation}
\label{massunder8pi1}
 \int_{ B_{\bar s}(0)}(\epsilon_n^2+|x|^2)^{\alpha_n}V_ne^{v_n}\,dx\leq 4\pi(1+\alpha_n)\tfrac{C_1+1}{C_1}<\min\{8\pi,8\pi(1+\alpha_\infty)\}.
\end{equation}
We notice that, since $y_n\to 0$ and $v_n(y_n)\to+\infty$, $0$ is a blow up point for $v_n$ in $B_{\bar s}(0)$. Also, by using the fact that $V_n\geq a>0$, we deduce that,
\[
 \int_{ B_{\bar s}(0)}(\epsilon_n^2+|x|^2)^{\alpha_n}e^{v_n}\,dx\leq 4\pi(1+\alpha_n)\tfrac{C_1+1}{aC_1}.
\]
Thus, by the Minimal Mass Lemma \ref{minimalmasslemma} for $v_n$ in $B_{\bar s}(0)$, we would have, possibly along a subsequence,
\[
 \int_{ B_{\rho}(0)}(\epsilon_n^2+|x|^2)^{\alpha_n}V_ne^{v_n}\,dx\geq \begin{cases}
            8\pi(1+\alpha_\infty), &\text{if}\,\,\al_\ii\in(-1,0), \\
            8\pi, &\text{if}\,\,\al_\ii\in[0,1),
            \end{cases}
\]
for every $\rho<\bar s$ small enough, which contradicts \eqref{massunder8pi1}. Thus necessarily $s_n\to0^+$.\\

Clearly, $v_n$ satisfies
\begin{align}\label{uninBsn}
 \begin{cases}
 -\D v_n=(\epsilon_n^2+|x|^2)^{\alpha_n}V_ne^{v_n}\qquad\text{in}\,\,B_{s_n}(y_n),\\
  \int_{B_{s_n}(y_n)}(\epsilon_n^2+|x|^2)^{\alpha_n}V_ne^{v_n}\,dy\leq 4\pi(1+\alpha_n)\tfrac{C_1+1}{C_1}<\min\{8\pi,8\pi(1+\alpha_\infty)\}.
 \end{cases}
\end{align}
At this point let us define,
\[
 t_n:=\max\{|y_n|,s_n\}\to0^+,\,\,\text{as}\,\,n\to+\infty.
\]
Here, we are naturally led to analyze three different cases:
\begin{itemize}
 \item (I) there exists a subsequence such that $\tfrac{\epsilon_n}{t_n}\to+\infty$, as $n\to+\infty$,
 \item (II) there exists a constant $C_1>0$ such that $\tfrac{\epsilon_n}{t_n}\leq C_1$, for all $n\in\N$ and, possibly along a subsequence,
 $\tfrac{|y_n|}{s_n}\to+\infty$, as $n\to+\infty$.
 \item (III) there exist constants $C_1>0$ and $C_2>0$ such that $\tfrac{\epsilon_n}{t_n}\leq C_1$ and $\tfrac{|y_n|}{s_n}\leq C_2$ for all $n\in\N$.
\end{itemize}

\medskip

Let us begin with Case (I) and consider first $\alpha_\infty\in[0,1)$.

\medskip

First of all notice that we also have $\tfrac{\epsilon_n}{s_n}\to+\infty$ and $\tfrac{\epsilon_n}{|y_n|}\to+\infty$, as $n\to+\infty$. Hence, we define
\begin{equation}
 \nonumber
  w_n(y):=v_n(y_n+s_ny)+2\log s_n+2\alpha_{n}\log \epsilon_n\qquad\text{in}\,\,B_1(0),
\end{equation}
which satisfies
\begin{align*}
 \begin{cases}
  -\D w_n=W_n(y)e^{\dis w_n}\qquad\text{in}\,\,B_1(0),\\
  W_n(y)=(1+|\tfrac{y_n}{\epsilon_n}+\tfrac{s_n}{\epsilon_n}y|^2)^{\alpha_n}V_n(y_n+s_n y)\to V(0)\qquad \text{in}\,\,C^0_{loc}(B_1(0)),\\
  \int_{B_1(0)}W_n(y)e^{\dis w_n}\,dy\leq 4\pi(1+\alpha_n)\tfrac{C_1+1}{C_1}<8\pi.
 \end{cases}
\end{align*}
Since $\tfrac{\epsilon_n}{s_n}\to+\infty$, for $n$ large we have
\[
 w_n(0)=v_n(y_n)+2\log s_n+2\alpha_{n}\log \epsilon_n\geq v_n(y_n)+2(1+\alpha_n)\log s_n,
\]
which, by \eqref{blowupsequenceatsn}, implies that $y=0$ is a blow up point for $w_n$ in $B_1(0)$. Then, we can apply the standard blow up analysis for Liouville equations (\cite{bm},\cite{ls}) to deduce that
\[
 \liminf\limits_{n\to +\ii}\int_{B_{s_n\rho}(y_n)}(\epsilon_n^2+|x|^2)^{\alpha_n}V_n e^{v_n}=\liminf\limits_{n\to +\ii}\int_{B_\rho(0)}W_n(y)e^{w_n}\,dy\geq 8\pi,
\]
for $\rho\in(0,1)$ small enough, which is a contradiction.

\medskip

Now, let us consider Case (I) with $\alpha_\infty\in(-1,0)$.

\medskip

Clearly we still have that $\tfrac{\epsilon_n}{s_n}\to+\infty$ and $\tfrac{\epsilon_n}{|y_n|}\to+\infty$, as $n\to+\infty$.
We notice first that $v_n$ satisfies,
\begin{align}
 \label{uninBsigman}
 \begin{cases}
 -\D v_n=(\epsilon_n^2+|x|^2)^{\alpha_n}V_ne^{\dis v_n}\qquad\text{in}\,\,B_{\sigma_n}(y_n),\\
  \int_{B_{\sigma_n}(y_n)}(\epsilon_n^2+|x|^2)^{\alpha_n}V_ne^{\dis v_n}\,dy\leq 4\pi\tfrac{1+C_1}{C_1}<8\pi.
 \end{cases}
\end{align}
Let us define the following quantities
\[
 \bar\delta_n:=e^{-\tfrac{v_n(y_n)}{2}}\to0^+,\qquad\qquad\tau_n:=\tfrac{\bar\delta_n}{\epsilon_n^{\alpha_n}}=\epsilon_n^{|\alpha_n|}\bar\delta_n\to0^+,
\]
and
\[
 d_n:=\tfrac{\sigma_n}{\tau_n}=\epsilon_n^{-|\alpha_n|}e^{\frac{1}{2}(v_n(y_n)+2\log(\sg_n))}\to+\infty,
\]
as $n\to+\infty$. Remark that, by \eqref{blowupsequenceatsn},
\[
 \tfrac{\tau_n}{\epsilon_n}=\tfrac{\bar\delta_n}{\epsilon_n^{1+\alpha_n}}=e^{-\frac{1}{2}(v_n(y_n)+2(1+\alpha_n)\log(\epsilon_n))}\leq
 e^{-\frac{1}{2}(v_n(y_n)+2(1+\alpha_n)\log(s_n))}\to 0^+,
\]
as $n\to+\infty$. Thus we define
\[
 u_n(y):=v_n(y_n+\tau_n y)+2\log \tau_n+2\alpha_n\log\epsilon_n=v_n(y_n+\tau_n y)-v_n(y_n),\qquad y\in B_{d_n}(0),
\]
which satisfies
\begin{align*}
 \begin{cases}
  -\D u_n= K_n(y)e^{\dis u_n}\qquad\text{in}\,\,B_{d_n}(0),\\
  K_n(y)=(1+|\tfrac{y_n}{\epsilon_n}+\tfrac{\tau_n}{\epsilon_n}y|^2)^{\alpha_n}V_n(y_n+\tau_n y)\to V(0)\qquad \text{in}\,\,C^0_{loc}(B_{d_n}(0)),\\
  \int_{B_{d_n}(0)}K_n(y)e^{\dis u_n}\,dy\leq 4\pi\tfrac{1+C_1}{C_1}<8\pi, \\
  u_n(y)\leq u_n(0)=0.
 \end{cases}
\end{align*}
It is readily seen that, for every $R>0$,
\[
 \int_{B_R(0)} e^{\dis u_n}\leq C,
\]
and then we can apply the concentration compactness alternative for Liouville equations (\cite{bm}) to deduce that $u_n$, up to a subsequence,
locally converges to $u$ in $C^0_{loc}(\R^2)$, where $u$ satisfies
\begin{align*}
 \begin{cases}
  -\D u=V(0)e^{\dis u}\qquad\text{in}\,\,\R^2,\\
   u(y)\leq u(0)=0, \,\,\text{for every}\,\,y\in\R^2.
 \end{cases}
\end{align*}
In particular, it is well-known (\cite{cl1}) that
\[
 V(0)\int_{\R^2}e^{\dis u}=8\pi.
\]
However,
\[
 \liminf\limits_{n\to+\infty}\int_{B_{\sigma_n}(y_n)}(\epsilon_n^2+|x|^2)^{\alpha_n}V_ne^{\dis v_n}\,dy=
 \liminf\limits_{n\to+\infty}\int_{B_{d_n}(0)}K_n(y)e^{\dis u_n}\,dy\geq V(0)\int_{\R^2}e^{\dis u}=8\pi.
\]
which is a contradiction to \eqref{uninBsigman}. This fact completes the study of Case (I).

\bigskip

Next, let us consider the Case (II) with $\alpha_\infty\in[0,1)$. Since $\frac{|y_n|}{s_n}\to +\ii$ then $s_n\leq|y_n|$ for $n$ large enough. Thus we have
\begin{equation}
 \label{epsilonyn}
 \frac{\epsilon_n}{|y_n|}\leq C,\,\,\forall\,n\in\N,
\end{equation}
for some $C>0$ and, up to a subsequence, $\tfrac{\epsilon_n}{|y_n|}\to\epsilon_0\geq0$. Recalling that $v_n$ satisfies \eqref{uninBsn}, let us define,
\begin{equation}
 \nonumber
  \bar w_n(y)=v_n(y_n+s_ny)+2\log s_n+2\alpha_{n}\log |y_n|\qquad\text{in}\,\,B_1(0),
\end{equation}
which satisfies
\begin{align*}
 \begin{cases}
  -\D \bar w_n=\bar W_n(y)e^{\dis \bar w_n}\qquad\text{in}\,\,B_1(0),\\
  \bar W_n(y)=(\tfrac{\epsilon_n^2}{|y_n|^2}+|\tfrac{y_n}{|y_n|}+\tfrac{s_n}{|y_n|}y|^2) ^{\alpha_n}V_n(y_n+s_n y)\to (1+\epsilon_0^2)^{\al_\ii}V(0)\qquad \text{in}\,\,C^0_{loc}(B_1(0)),\\
  \int_{B_1(0)}\bar W_n(y)e^{\dis \bar w_n}\,dy\leq 4\pi(1+\alpha_n)\tfrac{1+C_1}{C_1}<8\pi.
 \end{cases}
\end{align*}
Since $|y_n|>s_n$ for $n$ large enough, we have,
\[
 \bar w_n(0)=v_n(y_n)+2\log s_n+2\alpha_{n}\log |y_n|\geq v_n(y_n)+2(1+\alpha_n)\log s_n,
\]
which implies, by \eqref{blowupsequenceatsn}, that $0$ is a blow up point for $\bar w_n$ in $B_1(0)$. So we can apply the standard blow up analysis for Liouville equations (\cite{bm},\cite{ls}) to deduce that
\[
 \liminf\limits_{n\to+\infty}\int_{B_{s_n\rho}(y_n)}(\epsilon_n^2+|x|^2)^{\alpha_n}V_n e^{v_n}=
 \liminf\limits_{n\to+\infty}\int_{B_\rho(0)}\bar W_n(y)e^{\dis \bar w_n}\,dy\geq 8\pi,
\]
for $\rho\in(0,1)$ small enough, which is a contradiction.

\medskip

Next, let us consider Case (II) with $\alpha_\infty\in(-1,0)$.\\
Recall that \eqref{epsilonyn} holds true and then, up to a subsequence, we can assume that $\tfrac{\epsilon_n}{|y_n|}\to\epsilon_0\geq0$. We notice first that $v_n$ satisfies  \eqref{uninBsigman}, then we define the following quantities
\[
 \bar\delta_n:=e^{-\tfrac{v_n(y_n)}{2}}\to0^+\qquad\text{and}\qquad\bar \tau_n:=\tfrac{\bar\delta_n}{|y_n|^{\alpha_n}}=|y_n|^{|\alpha_n|}\bar\delta_n\to0^+,
\]
\[
 \bar d_n:=\tfrac{\sigma_n}{\bar\tau_n}=|y_n|^{-|\alpha_n|}e^{\frac{1}{2}(v_n(y_n)+2\log(\sg_n))}\to+\infty,
\]
as $n\to+\infty$. Since $\tfrac{|y_n|}{s_n}\to+\infty$, we have that
\[
 \tfrac{\tau_n}{|y_n|}=\tfrac{\bar\delta_n}{|y_n|^{1+\alpha_n}}=e^{-\frac{1}{2}(v_n(y_n)+2(1+\alpha_n)\log(|y_n|)}\leq e^{-\frac{1}{2}(v_n(y_n)+2(1+\alpha_n)\log(s_n))}\to0^+.
\]
as $n\to+\infty$. Thus we define
\[
 \bar v_n(y):=v_n(y_n+\bar\tau_n y)+2\log \bar\tau_n+2\alpha_n\log|y_n|=v_n(y_n+\bar\tau_n y)-v_n(y_n),\qquad y\in B_{\bar d_n}(0),
\]
which satisfies
\begin{align*}
 \begin{cases}
  -\D \bar v_n=\bar K_n(y)e^{\dis \bar v_n}\qquad\text{in}\,\,B_{\bar d_n}(0),\\
  \bar K_n(y)=(\tfrac{\epsilon_n^2}{|y_n|^2}+|\tfrac{y_n}{|y_n|}+\tfrac{\bar\tau_n}{|y_n|}y|^2)^{\alpha_n}V_n(y_n+\bar\tau_n y)\to
  (\epsilon_0^2+1)^{\alpha_\infty}V(0)\qquad \text{in}\,\,C^0_{loc}(B_{\bar d_n}(0)),\\
  \int_{B_{\bar d_n}(0)}\bar K_n(y)e^{\dis \bar v_n}\,dy\leq 4\pi\tfrac{1+C_1}{C_1}<8\pi, \\
  \bar v_n(y)\leq \bar v_n(0)=0.
 \end{cases}
\end{align*}
It is easy to see that, for every $R>0$,
\[
 \int_{B_R(0)} e^{\bar v_n}\leq C,
\]
then we can apply the concentration compactness alternative for Liouville equations (\cite{bm}) to deduce that $\bar v_n$, up to a subsequence, locally converges to $\bar v$ in $C^0_{loc}(\R^2)$, where $\bar v$ satisfies
\begin{align*}
 \begin{cases}
  -\D \bar v=(\epsilon_0^2+1)^{\alpha_\infty}V(0)e^{\dis \bar v}\qquad\text{in}\,\,\R^2,\\
  \bar v(y)\leq \bar v(0)=0, \,\,\text{for every}\,\,y\in\R^2.
 \end{cases}
\end{align*}
In particular, it is well-known (\cite{cl1}) that
\[
 (\epsilon_0^2+1)^{\alpha_\infty}V(0)\int_{\R^2}e^{\dis \bar v}=8\pi.
\]
Therefore we have,
\[
 \liminf\limits_{n\to+\infty}\int_{B_{\sigma_n}(y_n)}(\epsilon_n^2+|x|^2)^{\alpha_n}V_ne^{\dis v_n}\,dy=
 \liminf\limits_{n\to+\infty}\int_{B_{\bar d_n}(0)}\bar K_n(y)e^{\dis \bar v_n}\,dy\geq (\epsilon_0^2+1)^{\alpha_\infty}V(0)\int_{\R^2}e^{\dis \bar v}=8\pi,
\]
which is a contradiction to \eqref{uninBsigman}. This completes the study of Case (II).\\

At last we consider Case (III) with $\alpha\in(-1,1)$.\\
We notice first that,
\begin{align*}
 \tfrac{\epsilon_n}{s_n}=
 \begin{cases}
\tfrac{\epsilon_n}{t_n}, &\text{if}\,\,s_n\geq|y_n|\\
\tfrac{\epsilon_n}{t_n}\tfrac{|y_n|}{s_n} , &\text{if}\,\,s_n\leq|y_n|
\end{cases}\Bigg\}\leq \max\{C_1,C_1C_2\},
\end{align*}
for every $n\in\N$ and then, possibly along a subsequence, we can assume that,
\[
 \frac{\epsilon_n}{s_n}\to\epsilon_0,
\]
\[
 \xi_n=\frac{y_n}{s_n}\to\xi_0,
\]

for some $\epsilon_0\in[0,+\infty)$ and $\xi_0\in\R^2$. Recalling that $v_n$ satisfies \eqref{uninBsn}, we define
\[
 \widetilde w_n(y)=v_n(s_n y)+2(1+\alpha_n)\log s_n\qquad \text{in}\,\, B_{1}(\xi_n),
\]
which satisfies
\begin{align*}
 \begin{cases}
  -\D \widetilde w_n=\widetilde W_n(y)e^{\dis \widetilde w_n}\qquad\text{in}\,\,B_{\frac{1}{2}}(\xi_0),\\
  \widetilde W_n(y)=(\tfrac{\epsilon_n^2}{s_n^2}+|y|^2)^{\alpha_n}V_n(s_n y)\to (\epsilon_0^2+|y|^2)^{\alpha_n}V(0)
  \qquad\text{in}\,\,C^0_{loc}\left(B_{\frac{1}{2}}(\xi_0)\right),
  \\
  \int_{B_{\frac{1}{2}}(\xi_0)}\widetilde W_n(y)e^{\dis \widetilde w_n}\,dy\leq 4\pi(1+\alpha_n)\tfrac{1+C_1}{C_1}<\min\{8\pi,8\pi(1+\alpha_\ii)\}.
 \end{cases}
\end{align*}
 We notice that
 \[
  \widetilde w_n(\xi_n)= v_n(y_n)+2(1+\alpha_n)\log s_n,
 \]
which implies, by \eqref{blowupsequenceatsn}, that $\xi_0$ is a blow up point for $\widetilde w_n$ in $B_{\frac{1}{2}}(\xi_0)$, for $n$ large enough. \\
If either $\epsilon_0>0$ or $\xi_0\neq0$, we can apply the standard blow up analysis for Liouville equations (\cite{bm},\cite{ls}) to deduce that
\[
 \liminf\limits_{n\to +\ii}\int_{B_{s_n\rho}(y_n)}(\epsilon_n^2+|x|^2)^{\alpha_n}V_n e^{v_n}=
 \liminf\limits_{n\to +\ii}\int_{B_{\rho}(\xi_n)}\widetilde W_n(y)e^{\dis \widetilde w_n}\,dy\geq
\]
\[
 \int_{B_{\frac{\rho}{2}}(\xi_0)}\widetilde W_n(y)e^{\dis \widetilde w_n}\,dy\geq8\pi,
\]
for some $\rho\in(0,1)$ small enough, which is a contradiction. \\
On the other hand, if $|\xi_0|=\epsilon_0=0$, then we apply the Minimal Mass Lemma \ref{minimalmasslemma},
with $\widetilde \epsilon_n=\tfrac{\epsilon_n}{s_n}$, and we deduce once again that
\[
 \liminf\limits_{n\to +\ii}\int_{B_{s_n\rho}(y_n)}(\epsilon_n^2+|x|^2)^{\alpha_n}V_n e^{v_n}\geq\begin{cases}
            8\pi(1+\alpha_\infty), &\text{if}\,\,\al_\ii\in(-1,0), \\
            8\pi, &\text{if}\,\,\al_\ii\in[0,1),
            \end{cases}
\]
which is a contradiction.\\
This fact concludes the proof of Case (III).
\end{proof}

\bigskip

\section{The inequality \eqref{sup+inf.intro} is almost sharp.}\label{sec4}
It has ben proved in \cite{bcwyzOnsager}
that if $v_n$ is a sequence of solutions of \eqref{intro:eq} in $\om=B_1$, such that 
$$
\al_n=\frac{\lm_n}{4\pi}\sg,\quad\lm_n\to \lm_\ii\in \left(0,\frac{4\pi}{\sg}\right),\quad \al_\ii=\frac{\lm_\ii}{4\pi}\sg<1,\quad \sg\in \left(0,\frac12\right),
$$
which also satisfy, 
$$
\lm_n=\int\limits_{\om}\left({\epsilon_n^2+|x|^2}\right)^{\al_n} V_n e^{v_n},\quad \epsilon_n\to 0^+,
$$
$$
V_n\geq 0, \,\, V_n\to V \text{ uniformly in }\overline B_1\,\,\mbox{and in } C^{1}_{\rm loc}(B_1),
$$
and
$$
  \underset{\p B_1}\max\,v_n-\underset{\p B_1}\min\,v_n\leq C,
$$ 
and that $x=0$ is the unique blow up point for $v_n$ in $B_1$, that is,
\[
 \text{for any}\,\,r\in(0,1),\exists \, C_r>0, \,\,\text{such that}:
\]
\begin{equation*}
    \label{bounded}
\underset{\overline{B_1\backslash B_r}}{\max}\, v_n\leq C_r,
\end{equation*}
\begin{equation*}
    \label{explosion}
\underset{\overline{B_r}}\max\, v_n\to\infty,
\end{equation*}
then $V(0)>0$ and a new kind of blow up phenomenon takes place. We called it ``blow-up and concentration without quantization", where the lack of compactness of solutions comes with a concentration phenomenon but, unlike the classical regular (\cite{yy}) and singular cases (\cite{bt}), the corresponding mass is not quantized, being free to take values in the full interval $\lm_\ii\in (8\pi,\min\{\frac{8\pi}{1-2\sg},\tfrac{4\pi}{\sg}\})$. More exactly (see Theorem 3.3 in \cite{bcwyzOnsager}), we proved that there are only three possibilities, corresponding in fact to the three situations already discussed in the minimal mass Lemma above, in which cases we obtain the following global profiles. 
Let $r\leq \frac12$,
$$
v_n(x_n)=\underset{\overline{B_r}}
\max\, v_n\to+\infty, \quad |x_n|\to 0^+, 
$$
$$
\delta_n^{2(1+\al_n)}=e^{\dis -v_n(x_n)}\to 0^+,
$$
and 
\[
 t_n=\max\{\delta_n,|x_n|\}\to0^+.
\]
Then, either,
\begin{itemize}
 \item \mbox{\rm (I):}  there exists a subsequence such that $\tfrac{\epsilon_n}{t_n}\to+\infty$,
\end{itemize}
in which case we have $\lm_\ii=8\pi$ and
\begin{equation}\label{profilev:H}
v_n(x)=\log\left(\dfrac{e^{\dis v_n(x_n)}}
{\left(1+\gamma_n\theta_n^{2(1+\bns)}\epsilon_n^{-\frac{\lm_n}{4\pi}}|x-{x_n}|^{\frac{\lm_n}{4\pi}}\right)^2}\right)+O(1),\quad z\in B_{r}(0);
\end{equation}
where
$\theta_n^{2(1+\al_{n})}=
\left(\tfrac{\epsilon_n}{\delta_n}\right)^{2(1+\al_{n})}\to +\ii$, $8 \gamma_n={(1+|\frac{x_n}{\epsilon_n}|^2)^{\bns} V_n({x_n})}$, or
\begin{itemize}
 \item \mbox{\rm (II):} there exists a subsequence such that $\tfrac{\epsilon_n}{t_n}\leq C$ and $\tfrac{|x_n|}{\delta_n}\to+\infty$,
\end{itemize}
in which case we have $\lm_\ii=8\pi$ and
\begin{equation}\label{profilev:H1}
v_n(x)=\log\left(\dfrac{e^{\dis v_n(x_n)}}
{\left(1+\bar\gamma_n\bar \theta_n^{2(1+\al_{n})}|x_n|^{-\frac{\lm_n}{4\pi}} |x-{x_n}|^{\frac{\lm_n}{4\pi}}\right)^2}\right)+O(1),\quad z\in B_{r}(0),
\end{equation}
where $\bar \theta_n^{2(1+\al_{n})}=\left(\tfrac{|x_{n}|}{\delta_n}\right)^{2(1+\al_{n})}\to +\ii$,
$8\bar\gamma_n={((\tfrac{\epsilon_n}{|x_n|})^2+1)^{\bns}}V_n(x_n)$, or
\begin{itemize}
 \item \mbox{\rm (III):} there exists a subsequence such that $\tfrac{\epsilon_n}{t_n}\leq C$ and
 $\tfrac{|x_n|}{\delta_n}\leq C$,
\end{itemize}
in which case, along a further subsequence if necessary, we have 
$$
\tfrac{\epsilon_n}{\delta_n}\to \epsilon_0\geq 0, \quad\tfrac{x_n}{\delta_n}\to y_0\in\R^2, $$ 
and $\lm_\ii\in (8\pi,\tfrac{8\pi}{1-2\sg}]$, if $\sigma\in(0,\tfrac{1}{4})$, or $\lm_\ii\in(8\pi,\tfrac{4\pi}{\sigma})$, if $\sg\in[\tfrac{1}{4},\tfrac12)$ and
\begin{equation}\label{profilevtilde:H1-IIb}
 v_n(x)=v_n(x_n)+ \widetilde U_n(\delta_n^{-1}x)+O(1), \quad x\in B_{r}(0),
\end{equation}

where
\begin{equation*}
    \widetilde U_n(y)=\graf{\widetilde w(y)+O(1),\quad |y|\leq R \\ -\frac{\lm_n}{2\pi}\log(|y|) + O(1),\quad R\leq  |y|\leq r\delta_n^{-1}}
\end{equation*}
and $\widetilde w$ is the unique solution of 
\begin{align*}
\begin{cases}\label{profile:tildew}
 -\D \widetilde w=\widetilde V(y) e^{\dis \widetilde w} \quad\text{in}\quad \R^2 \\
\widetilde V(y)=V(0)\left(\epsilon_0^2+|y|^2\right)^{\al_\ii},\\
\lm_\ii = \int_{\R^2}\widetilde V e^{\dis \widetilde w},\\
\widetilde w(y)\leq \widetilde w(y_0)=0.
\end{cases}
\end{align*}

It is still under investigation whether or not solutions of \eqref{intro:eq} satisfying either \eqref{profilev:H} or \eqref{profilev:H1} really exist. On the other hand, it is reasonable to think that, possibly under general assumptions, one can prove the existence of solutions satisfying \eqref{profilevtilde:H1-IIb}. 

\smallskip

However, it is interesting to test the validity of the ``$\sup+C\inf$" inequality \eqref{intro:sup+Cinf} for these kind of approximated solutions.

Concerning the sequences in (I), let us consider any $\mu=\tfrac{1}{C_1}<\frac{1-\al_\ii}{1+\al_\ii}$, then we have that
\begin{align*}
\mu\, v_n(x_n)+\inf\limits_{\pa B_1}v_n&=O(1)+\log(\delta_n^{-2(1+\bns)\mu})+\log\left(\tfrac{\delta_n^{-2(1+\bns)}}{\theta_n^{4(1+\bns)}\epsilon_n^{-\frac{\lm_n}{2\pi}}}\right)\\
&=O(1)+\log\left(\theta_n^{2(1+\bns)(\mu-1)}\epsilon_n^{\frac{\lm_n}{2\pi}-2(1+\bns)(1+\mu)}\right).
\end{align*}
We notice that $\mu<1$, therefore $2(1+\bns)(\mu-1)<0$, for every $n\in\N$, and, by using the fact that $\lm_n=8\pi+o(1)$,
\begin{align*}
\frac{\lm_n}{2\pi}-2(1+\bns)(1+\mu)&=-2(1+\bns)\left(\mu-\tfrac{\frac{\lm_n}{2\pi}-2(1+\bns)}{2(1+\bns)}\right)=-2(1+\bns)\left(\mu-\tfrac{1-\bns}{1+\bns}+o(1)\right)\geq0,
\end{align*}
for any $n$ large enough. Thus,
\begin{align*}
\mu\, &v_n(x_n)+\inf\limits_{\pa B_1}v_n=\\
&=O(1)+2(1+\bns)(\mu-1)\log(\theta_n)-2(1+\bns)\left(\mu-\tfrac{1-\bns}{1+\bns}+o(1)\right)\log\epsilon_n\to-\infty.
\end{align*}
On the other hand, by taking $C_1=\mu=1$, we have that 
$$
v_n(x_n)+\inf\limits_{\pa B_1}v_n=-2(1+\bns)\left(1-\tfrac{1-\bns}{1+\bns}+o(1)\right)\log(\epsilon_n)=(-4\al_n+o(1))\log(\epsilon_n)\to +\ii.
$$
The sequences in (II) share the same behavior, therefore we skip the details. \\
For the class of solutions (I) and (II), it is not clear what happens by peaking $C_1\in\left(1,\frac{1+\al_\ii}{1-\al_\ii}\right]$, while for sure they fail to satisfy \eqref{sup+inf.intro} with $C_1=1$.

\medskip

At last, concerning the solutions in (III), by recalling that $\al_\ii=\frac{\lm_\ii}{4\pi}\sg<1$, we show that, if $C_1\geq\frac{1+\al_\ii}{1-\al_\ii}$, then 
\begin{equation}\label{sup+infholds}
    \sup\limits_{B_r} v_n+C_1\inf\limits_{B_1}v_n=
v_n(x_n)+C_1\inf\limits_{\pa B_1}v_n\to -\ii.
\end{equation}
Thus, putting $\frac{1}{C_1}=\tau\frac{1-\al_\ii}{1+\al_\ii}$ for some $\tau\leq1$, and peaking $\lm_n=8\pi \mu_n$ for some $\mu_n\to \mu>1$, we have, 
\begin{align*}
\frac{1}{C_1} v_n(x_n)+\inf\limits_{\pa B_1}v_n&=\frac{2(1+\al_n)}{C_1}\log(\dt_n^{-1})+2(1+\al_n)\log(\dt_n^{-1})-\frac{\lm_n}{2\pi}\log(\dt_n^{-1})+O(1)\\
&=2(1+\al_n)\left(\frac{1}{C_1}+1-\frac{\lm_n}{4\pi(1+\al_n)}\right)\log(\dt_n^{-1})+O(1)\\
&=2(1+\al_n)\left(\tau\frac{1-\al_\ii}{1+\al_\ii}+1-\frac{2\mu_n}{1+\al_n}\right)\log(\dt_n^{-1})+O(1)\\
&=2(1+\al_n)\left(\tau\frac{1-\al_\ii}{1+\al_\ii}+1-
\frac{2\mu}{1+\al_\ii}(1+o(1))\right)\log(\dt_n^{-1})+O(1).
\end{align*}
Now, since $\tau\leq1$ and $\mu>1$, we notice that
\begin{align*}
\tau\frac{1-\al_\ii}{1+\al_\ii}+1-
\frac{2\mu}{1+\al_\ii}(1+o(1))&\leq \frac{1-\al_\ii}{1+\al_\ii}+1-
\frac{2\mu}{1+\al_\ii}(1+o(1))\\
&=
\frac{2-2\mu(1+o(1))}{1+\al_\ii}\leq0,
\end{align*}
therefore \eqref{sup+infholds} easily follows.

On the other hand, we show that, if $C_1<\frac{1+\al_\ii}{1-\al_\ii}$, then 
\begin{equation}\label{nosup+inf}
    \sup\limits_{B_r} v_n+C_1\inf\limits_{B_1}v_n=
v_n(x_n)+C_1\inf\limits_{\pa B_1}v_n\to +\ii.
\end{equation}

In fact, remark that in (III) we are allowed to take in principle any value of $\lm_\ii$ as close as we wish to $8\pi^+$, see \cite{Lin1}. Thus, putting $\frac{1}{C_1}=\tau\frac{1-\al_\ii}{1+\al_\ii}$ for some $\tau >1$, and peaking $\lm_n=8\pi \mu_n$ for some $\mu_n\to \mu>1$ to be fixed later on, we have, 
\begin{align*}
\frac{1}{C_1} v_n(x_n)+\inf\limits_{\pa B_1}v_n&=2(1+\al_n)\left(\tau\frac{1-\al_\ii}{1+\al_\ii}+1-\frac{2\mu_n}{1+\al_n}\right)\log(\dt_n^{-1})+O(1)\\
&=2(1+\al_n)\left(\tau\frac{1-\al_\ii}{1+\al_\ii}-
\frac{2\mu-1-\al_\ii}{1+\al_\ii}(1+o(1))\right)\log(\dt_n^{-1})+O(1)
\end{align*}
and \eqref{nosup+inf} follows since, as $\mu\to 1^+$, the quantity in parenthesis converges to 
$$
\frac{1-\al_\ii}{1+\al_\ii} (\tau-1 +o(1)).
$$ 

\bigskip

\section{Appendix: Known results}
We need the following well known result, see \cite{Lin1} and \cite{llty}.
\begin{lemma}\label{massstrange}
 Let $\psi$ be a solution of the following problem
 \begin{equation*}
 \begin{cases}
 -\D \psi(x)= (\epsilon_0^2+|x|^2)^{\alpha}e^{\dis \psi}\quad\text{in}\,\, \R^2, \\
 \underset{\R^2}\int(\epsilon_0^2+|x|^2)^{\alpha}e^{\dis \psi}<+\infty,
 \end{cases}
\end{equation*}
for some $\epsilon_0\geq0$ and $\alpha>-1$. Let us denote by
\begin{equation*}
 \beta=\tfrac{1}{2\pi}\underset{\R^2}\int(\epsilon_0^2+|x|^2)^{\alpha}e^{\dis \psi}.
\end{equation*}
Then:
\begin{itemize}
 \item If $-1<\alpha< 0$, we have,
 \begin{equation*}
 \label{massnegative}
  8\pi(1+\alpha)\leq \underset{\R^2}\int(\epsilon_0^2+|x|^2)^{\alpha}e^{\dis \psi}<8\pi,
 \end{equation*}
 where the l.h.s equality holds true if and only if $\epsilon_0=0$;
\item If $\alpha\geq 0$, then
 \begin{equation*}
 \label{masspositive}
\max\{8\pi,4\pi(1+\alpha)\}\leq \underset{\R^2}\int(\epsilon_0^2+|x|^2)^{\alpha}e^{\dis \psi}\leq  8\pi(1+\alpha),
 \end{equation*}
 where the l.h.s. equality holds true if and only if $\alpha=0$ while the right hand side equality holds true if and only if $\epsilon_0=0$.
\end{itemize}
\end{lemma}

\section*{Acknowledgements}

Daniele Bartolucci and Paolo Cosentino were partially supported by the MIUR Excellence Department Project MatMod@TOV awarded to the Department of Mathematics, University of Rome ``Tor Vergata'', CUP E83C23000330006, by PRIN project 2022 2022AKNSE4, ERC PE1\_11, ``{\em Variational and Analytical aspects of Geometric PDEs}'', by the E.P.G.P. Project sponsored by the University of Rome ``Tor Vergata'', E83C25000550005, and by INdAM-GNAMPA Project, CUP E53C25002010001. They are members of the INDAM Research Group ``Gruppo Nazionale per l’Analisi Matematica, la Probabilit\`a e le loro Applicazioni".\\
Lina Wu was partially supported by the National Natural Science Foundation of China (12201030).

\maketitle

\end{document}